\documentclass[11pt]{amsart}
\usepackage{tgpagella}
\usepackage{euler}
\usepackage[T1]{fontenc}
\usepackage{amsmath, amssymb}
\usepackage[hidelinks]{hyperref}
\usepackage[english]{babel}
\usepackage{mathrsfs}
\usepackage{eucal}
\usepackage[all]{xy}
\newtheorem{thm}{Theorem}[section]

\newtheorem{defn}[thm]{Definition}
\newtheorem{lemma}[thm]{Lemma}
\newtheorem{prop}[thm]{Proposition}
\newtheorem{cor}[thm]{Corollary}
\newtheorem{rmk}[thm]{Remark}
\newtheorem{example}[thm]{Example}
\newtheorem{conj}[thm]{Conjecture}
\newtheorem{fact}[thm]{Fact}
\newtheorem{question}[thm]{Question}
\newcommand\e\epsilon

\def\Con{\operatorname{Con}}
\newcommand{\cstar}{$\mathrm{C}^*$}

\makeatletter

\def\indsym#1#2{%
  \setbox0=\hbox{$\m@th#1x$}%
  \kern\wd0%
  \hbox to 0pt{\hss$\m@th#1\mid$\hbox to 0pt{$\m@th#1^{#2}$}\hss}%
  \lower.9\ht0\hbox to 0pt{\hss$\m@th#1\smile$\hss}%
  \kern\wd0}

\def\nindsym#1#2{%
  \setbox0=\hbox{$\m@th#1x$}%
  \kern\wd0%
  \hbox to 0pt{\hss$\m@th#1\not$\kern1.4\wd0\hss}
  \hbox to 0pt{\hss$\m@th#1\mid$\hbox to 0pt{$\m@th#1^{\,#2}$}\hss}%
  \lower.9\ht0\hbox to 0pt{\hss$\m@th#1\smile$\hss}%
  \kern\wd0}

\def\dotminussym#1#2{%
  \setbox0=\hbox{$\m@th#1-$}%
  \kern.5\wd0%
  \hbox to 0pt{\hss\hbox{$\m@th#1-$}\hss}%
  \raise.6\ht0\hbox to 0pt{\hss$\m@th#1.$\hss}%
  \kern.5\wd0}
\newcommand{\dotminus}{\mathbin{\mathpalette\dotminussym{}}}

\def \Th{\operatorname{Th}}
\def \R{\mathcal R}

\def \d{\mathbf{d}}

\def \rest{\operatorname{rest}}
\def \Sent{\operatorname{Sent}}
\def \Theorem{\operatorname{Theorem}}
\def \Pre{\operatorname{Pre-Cond}}
\def \Cond{\operatorname{Cond}}
\title{Operator algebras with hyperarithmetic theory}
\author{Isaac Goldbring and Bradd Hart}
\thanks{I. Goldbring was partially supported by NSF CAREER grant DMS-1349399. B. Hart was supported by an NSERC Discovery Grant}

\address{Department of Mathematics\\University of California, Irvine, 340 Rowland Hall (Bldg.\# 400),
Irvine, CA 92697-3875}
\email{isaac@math.uci.edu}
\urladdr{http://www.math.uci.edu/~isaac}

\address{Department of Mathematics and Statistics, McMaster University, 1280 Main St., Hamilton ON, Canada L8S 4K1}
\email{hartb@mcmaster.ca}
\urladdr{http://ms.mcmaster.ca/~bradd/}

\begin{document}

\begin{abstract}
We show that the following operator algebras have hyperarithmetic theory:  the hyperfinite II$_1$ factor $\mathcal R$, $L(\Gamma)$ for $\Gamma$ a finitely generated group with solvable word problem, $C^*(\Gamma)$ for $\Gamma$ a finitely presented group, $C^*_\lambda(\Gamma)$ for $\Gamma$ a finitely generated group with solvable word problem, $C(2^\omega)$, and $C(\mathbb P)$ (where $\mathbb P$ is the pseudoarc).  We also show that the Cuntz algebra $\mathcal O_2$ has a hyperarithmetic theory provided that the Kirchberg embedding problems has an affirmative answers.  Finally, we prove that if there is an existentially closed (e.c.) II$_1$ factor (resp. \cstar-algebra) that does not have hyperarithmetic theory, then there are continuum many theories of e.c. II$_1$ factors (resp. e.c. \cstar-algebras).
\end{abstract}

\maketitle

\section{Introduction}


It is a typical question in logic to ask whether or not a given mathematical structure is \emph{decidable}, that is, whether or not there is an algorithm that, upon input a first-order sentence in the language of that structure, can decide whether or not the sentence is true or false.  For example, theorems of Tarski establish that the real and complex fields are decidable, while the famous G\"odel Incompleteness Theorem states that the ring of integers is undecidable.

When a structure is undecidable or one cannot determine its decidability, it is also common to instead prove results that say that the structure is no more complicated than another structure (in the sense of Turing reduction; see Section 2 below).  It is the purpose of this paper to prove results along these lines for various familiar operator algebras.  Temporarily, let us say that an operator algebra is \emph{hyperarithmetic} if it is no more complicated than the ring of integers.  (Technically speaking, this is a fairly low level hyperarithmetic set and the precise definitions will be given in Section 2.)

In this paper, we prove that the following operator algebras are hyperarithmetic:

\begin{enumerate}
\item the hyperfinite II$_1$ factor $\mathcal R$;
\item $L(\Gamma)$ for $\Gamma$ a finitely generated group with solvable word problem;
\item $C^*(\Gamma)$ for $\Gamma$ a finitely presented group;
\item $C^*_\lambda(\Gamma)$ for $\Gamma$ a finitely generated group with solvable word problem;
\item $C(2^\omega)$;
\item $C(\mathbb P)$ (where $\mathbb P$ is the pseudoarc).
\item the Cuntz algebra $\mathcal O_2$ (provided every \cstar-algebra embeds into an ultrapower of $\mathcal O_2$).
\end{enumerate}

To ensure that the above results are nontrivial, we note that the main result of \cite{BCI} implies that there are continuum many different theories of II$_1$ factors (resp. \cstar-algebras) whence there can be only countably many hyperarithmetic II$_1$ factors and \cstar-algebras.  

Our results are proven using two main techniques.  The first is to quote a result of Camrud, McNicholl, and the current author from \cite{CGM}, which states that if an algebra has a \emph{hyperarithmetic presentation}, then it is hypearithmetic; this technique applies to examples (1)-(4) from the above list (where the results of \cite{canyou} are used to show that the algebras have hyperarithmetic presentations).  

Our other main technique is to prove the continuous analogue of a classical result (see \cite[Chapter 7, Section 2]{HW}), which shows that if $T$ is an arithemetic theory, then its \emph{finite forcing companion} is hyperarithmetic.  This technique is how we deduce the remaining examples.

It is an open problem in the model theory of operator algebras whether or not there are two distinct theories of \emph{existentially closed} (e.c.) II$_1$ factors (resp. \cstar-algebras).  Here, an e.c. operator algebra is the model-theoretic formulation for what it means for the operator algebra to be ``algebraically closed'' in the sense of field theory.  In this paper, we prove the continuous analogue of a fact from classical logic due to Simpson (\cite[Chapter 7, Section 4]{HW}), which, when specialized to the above context, states if there is an e.c. II$_1$ factor (resp. \cstar-algebra) that is not hyperarithmetic, then there are actually continuum many distinct theories of e.c. II$_1$ factors (resp. \cstar-algebras).  In the case of groups, one can show that the theory of the so-called \emph{infinitely generic} groups interprets second-order arithmetic (so is not hyperarithmetic), leading to a quick proof that there are continuum many distinct theories of e.c. groups.  We conjecture that the same fact should be true in the II$_1$ factor and \cstar-algebra contexts.

The appropriate logical framework for studying operator algebras is \emph{continuous logic} and there have been a plethora of attempts to approach this logic with operator algebraists in mind (e.g. \cite{munster}).  For this reason, and to keep the paper relatively short, we assume that the reader is somewhat familiar with continuous model theory.

However, we do not assume that the reader is familiar with the requisite computability theory.  For that reason, in Section 2 we give a rapid account of the basic facts needed to follow the proofs in this paper.  We also take this opportunity to treat G\"odel numbering in continuous logic and to discuss some subtleties about relative computability of theories in continuous logic.

We would like to thank Timothy McNicholl and Andreas Thom for helpful discussions around this work.

%

\section{Preliminaries}

\subsection{Computability theory}

We begin this section with a brief discussion of the basic notions from computability theory; a good and accessible reference is \cite{Enderton} and a more advanced reference is \cite{Soare}.

Computability theory studies the question of what it means for a function $f:\mathbb N^k\to \mathbb N$ to be ``computable.''  Na\"ively speaking, such a function $f$ should be computable if there is an algorithm\footnote{This algorithm is allowed to refer to basic arithmetic operations such as addition and multiplication.} such that, upon input $(a_1,\ldots,a_k)\in \mathbb N^k$, runs and eventually halts, outputting the result $f(a_1,\ldots,a_k)$.  There are many approaches to formalizing this heuristic (e.g. Turing-machine computable functions and recursive functions) and all known formalizations can be proven to yield the same class of functions.  This latter fact gives credence to the \emph{Church-Turing thesis}, which states that this aformentioned class of functions is indeed the class of functions that are computable in the na\"ive sense described above.  In the rest of this paper, we will never argue about this class of functions using any formal definition but will only argue informally in terms of some kind of algorithm or computer; this is often referred to as arguing using the Church-Turing Thesis.  We call a subset of $\mathbb N^k$ computable if its characteristic function is computable. 

We will also consider a relative form of computability in the following sense:  given a function $g:\mathbb N^l\to \mathbb N$, a \textbf{$g$-algorithm} is an algorithm in the usual sense that is also allowed to make ``queries'' to an ``oracle'' which has access to the function $g$.  Thus, for example, in the course of running the algorithm, the computer is allowed to query the oracle and ask what is the value of $g(15)$ and then use that information in the computation.  A function $f$ as above is then said to be \textbf{$g$-computable} (or that \textbf{$g$ computes $f$}) if it is computable using a $g$-algorithm.  Of course, if $g$ were itself computable, then $f$ would be computable as well.  We say that a subset of $\mathbb N^k$ is $g$-computable if its characteristic function is $g$-computable.

If $f:\mathbb N^k\to \mathbb N$ and $g:\mathbb N^l\to \mathbb N$ are functions, we say that $f$ \textbf{Turing reduces} to $g$, denoted $f\leq_T g$, if $f$ is $g$-computable.  This is a preorder on the set of functions from finite cartesian powers of $\mathbb N$ to $\mathbb N$.  We say that $f$ and $g$ as above are \textbf{Turing equivalent}, denoted $f\equiv_T g$, if $f\leq_T g$ and $g\leq_T f$.  This is indeed an equivalence relation whose equivalence classes are called \textbf{Turing degrees}.  We typically denote Turing degrees by $\d$.  We often abuse notation and refer to a $\d$-algorithm when we mean a $g$-algorithm for some $g\in \d$.  (The choice of representative of $\d$ is often irrelevant).  Similarly, we speak of $\d$-computable functions and sets.  The preorder $\leq_T$ induces a partial order, also denoted $\leq_T$, on the set of Turing degrees.  We write $\d_1<_T \d_2$ if $\d_1\leq_T \d_2$ but $\d_1\not=\d_2$.  The Turing degree consisting of precisely the computable functions is denoted by $\mathbf 0$.  It is clear that $\mathbf 0$ is the minimum element of the poset of Turing degrees.  

Given a Turing degree $\d$, there is a unique Turing degree $\d'$, called the \textbf{jump of $\d$}, such that $\d<_T \d'$ and whenever $\d<_T \mathbf e$, then $\d'\leq_T \mathbf e$.  ($\d'$ is the Turing degree of the \emph{halting problem} relativized to $\d$.)  We can iterate the jump operation a finite number of times and we let $\mathbf d^{(n)}$ denote the $n^{\text{th}}$ iterated jump of $\d$.  We will shortly discuss how to iterate the jump an infinite number of times as well.

It will become convenient for us to use an equivalent formulation of these concepts in terms of definability in first-order arithmetic.  We consider the first-order language for studying the natural numbers that contains function symbols for addition and multiplication, a relation symbol for the ordering, and constant symbols for $0$ and $1$.  One defines a hierarchy of formulae, called the \textbf{arithmetic hierarchy}, first by defining $\Delta_0$ formulae to consist of only those formulae defined using ``bounded'' quantifiers (such as $\forall x<m$ and $\exists x<m$).  Supposing by recursion that one has defined $\Sigma_m$ and $\Pi_m$ formulae, one defines $\Sigma_{m+1}$ formulae to be those of the form $\exists x \varphi$, where $\varphi$ is $\Pi_m$, and $\Pi_{m+1}$ formulae to be of the form $\forall x\psi$, where $\psi$ is $\Sigma_m$.  A set $A\subseteq \mathbb N$ is called $\Sigma_m$ or $\Pi_m$ if it can be defined by a formula of that type, and it is called $\Delta_m$ if it can be defined both by a $\Sigma_m$ formula and a $\Pi_m$ formula.  

Since the $\Delta_0$ formulae are precisely those that define computable sets, given a degree $\d$, if we change the definition of $\Delta_0$ formulae above to consider those formulae which define $\d$-computable sets, and then continue the recursive definitions as before, one arrives at the relativized hierachy of formulae denoted $\Sigma_n^\d$, $\Pi_n^\d$, and $\Delta_n^\d$.  We will often use:

\begin{fact}[Post's Theorem]
If $\d$ is a Turing degree and $A\subseteq \mathbb N$, then $A$ is $\Sigma_{n+1}^\d$ if and only if it is $\Sigma_1^{\d^{(n)}}$.  Consequently, $A$ is $\Delta_{n+1}^\d$ if and only if $A\leq_T \mathbf \d^{(n)}$.  
\end{fact}

A set is called \textbf{arithmetical} if it is $\Sigma_n$ for some $n$.  Thus, a consequence of Post's theorem is that a set is arithmetical if and only if it is Turing reducible to $\mathbf 0^{(n)}$ for some $n$.

As referred to in the title of this paper, we need to work with so-called hyperarithmetic sets, which is a class of sets properly extending the class of arithmetic sets.  We first present this class in the spirit of the previous discussion.  This requires us to move pass first-order logic and into second-order logic, where we now study natural numbers using a logic which, besides the usual first-order expressive power, also allows one to quantify over subsets of natural numbers and which includes a relation symbol for membership between natural numbers and sets of natural numbers.  We typically use lowercase letters for variables ranging over numbers and uppercase letters for variables ranging over sets of natural numbers.  

An \textbf{arithmetical} formula is one whose only quantifiers are number quantifiers.  A $\Sigma_1^1$ formula is one of the form $\exists X_1\cdots \exists X_n \varphi$, where $\varphi$ is an arithmetic formula.  A subset of $\mathbb N$ (or $\mathcal P(\mathbb N)$) is called $\Sigma_1^1$ if it is defined by a $\Sigma_1^1$ formula.  $\Pi_1^1$ formulae are defined in same way using universal set quantifiers.  A subset of $\mathbb N$ is called \textbf{hyperarithmetic} if it is both $\Sigma_1^1$ and $\Pi_1^1$.

There is a way to view hyperarithmetic sets in a manner analogous to Post's theorem.  Without going into too many details, there is a countable ordinal $\omega_1^{CK}$ such that, for all ordinals $\gamma<\omega_1^{CK}$, one can define the $\gamma^{\text{th}}$ jump of $\mathbf 0$, denoted $\mathbf 0^{(\gamma)}$.  (See \cite{sacks} for the details.) When $\gamma\in \omega$, this jump operation agrees with the earlier notion denoting finitely many iterated jumps.  One can then show that a subset of $\mathbb N$ is hyperarithmetic if and only if it Turing reduces to $\mathbf 0^{(\gamma)}$ for some $\gamma<\omega_1^{CK}$.  One can also relativize this procedure and speak of $\d^{(\gamma)}$ for arbitrary Turing degrees $\d$ and ordinals $\gamma<\omega_1^\d$.

Of particular interest to us is the ``smallest'' hyperarithmetic, non-arithmetic degree, namely $\mathbf 0^{(\omega)}$.  One can show that the Turing degree of the set of G\"odel numbers for true first-order sentences of arithmetic is precisely $\mathbf 0^{(\omega)}$.

\subsection{G\"odel numbering in continuous logic}

Throughout this paper, $L$ denotes a continuous first-order language.  As in the classical case, it makes sense to speak of whether or not $L$ is \textbf{computable}.  Roughly speaking, $L$ is computable if there is an enumeration of the symbols for which one can computably tell what kind of symbol a given number represents, its arity, and, new to the continuous case, in case of function and predicate symbols, one can also uniformly computably calculate moduli of uniform continuity of the symbols.  A precise definition can be found in \cite{FMcN}.  Suffice it to say, the languages that occur in the operator algebra context are all computable.  We thus make the convention that, \emph{in the rest of this paper, $L$ denotes a computable continuous language}.

The set of \textbf{restricted} $L$-formulae is the smallest collection of $L$-formulae containing the atomic formulae and closed under the unary connectives $x\mapsto 0$, $x\mapsto 1$, and $x\mapsto \frac{x}{2}$, the binary connective $(x,y)\mapsto x\dotminus y:=\max(x-y,0)$, and the quantifiers $\sup_x$ and $\inf_x$.  We let $\operatorname{Form}_L$ and $\Sent_L$ denote the set of restricted $L$-formulae and restricted $L$-sentences respectively.

Since $L$ is computable, one can assign G\"odel numbers to restricted $L$-formulae and we let $\ulcorner \varphi\urcorner$ denote the G\"odel number of $\varphi$.  We let $\ulcorner \operatorname{Form}_L\urcorner$ and $\ulcorner \Sent_L\urcorner$ denote the sets of G\"odel numbers of restricted $L$-formulae and restricted $L$-sentences respectively.  If $k\in \ulcorner \operatorname{Form}_L\urcorner$, we let $\varphi_k$ denote the formula that it codes.  Every restricted $L$-formula is equivalent to one in prenex normal form and we use the usual notation $\forall_n$ and $\exists_n$ here as in the classical case.  Also as in the classical case, for each $n$, the set of G\"odel numbers of $\forall_n$ and $\exists_n$ $L_{\rest}$-sentences are computable; we refer to these sets as $\ulcorner \forall_{n}\text{-}\Sent_L\urcorner$ and $\ulcorner \exists_n\text{-}\Sent_L\urcorner$ respectively.  We let $\varphi_{\Sent_L}(x)$ be an arithmetical formula defining $\ulcorner \Sent_L\urcorner$ and similarly for $\varphi_{\forall\text{-}\Sent_L}(x)$, $\varphi_{\exists\text{-}\Sent_L}(x)$, $\varphi_{\operatorname{qf}_L}(x)$, the latter of which defines the codes of quantifier-free $L$-formulae.  We also let $\Phi_{\Sent_L}(X):=\forall p(p\in X \leftrightarrow \varphi_{\Sent_L}(p))$, another arithmetic formula.  If $\Phi_{\Sent_L}(X)$ holds, we set $T_X:=\{\sigma\in \Sent_L \ : \ \ulcorner\sigma\urcorner\in X\}$.

The following lemma is obvious:

\begin{lemma}\label{fandg}
There are recursive functions $f,g:\mathbb N^2\to \mathbb N$ such that, if $p,q\in \ulcorner \operatorname{Form}_L\urcorner$, then $f(p,n)=\ulcorner \varphi_p\dotminus 2^{-n}\urcorner$ and $g(p,q)=\ulcorner \varphi_p\dotminus \varphi_q\urcorner$.
\end{lemma}


Occasionally we will need to consider a set $C:=\{c_1,c_2,\ldots\}$ of new constant symbols and we set $L(C):=L\cup C$.  In this case, we can also assume that $L(C)$ is computable and that one can computably tell whether or not an element of $\ulcorner \operatorname{Form}_{L(C)}\urcorner$ actually belongs to $\ulcorner \operatorname{Form}_L\urcorner$ or not.

\subsection{Theories}

By an \textbf{$L$-theory}, we simply mean a subset of $\Sent_L$.  If $T$ is an $L$-theory, we set $\ulcorner T\urcorner:=\{k\in \ulcorner \Sent_L\urcorner \ : \ \varphi_k\in T\}$.  A theory $T$ is \textbf{complete} if, for each $L$-sentence $\sigma$, there is a unique $r$ such that $|\sigma-r|\in T$.

Given an $L$-structure $M$, we define the \textbf{theory of $M$} to be the function $\Th(M):\ulcorner \Sent_L\urcorner \to \mathbb R$ defined by $\Th(M)(k):=\varphi_k^M$.  As it stands, the theory of $M$ is not actually a theory in the above sense of the word.  However, if one considers the zeroset of $\Th(M)$, that is, the set $\{k \ : \ \varphi_k^M=0\}$, then one gets a theory in the above sense (after identifying $k$ with $\varphi_k$).  Under this identification, the complete theories are precisely those of the form $\Th(M)$ for some structure $M$.  

This paper is all about relative computability of theories.  Given the above two paragraphs, there is a subtlety that needs to be explained.  First, a definition:

\begin{defn}
Suppose that $\d$ is a Turing degree.
\begin{enumerate}
\item If $T$ is an $L$-theory, we say that $T$ is \textbf{$\d$-computable} if $\ulcorner T\urcorner$ is a $\d$-computable subset of $\mathbb N$.
\item If $M$ is an $L$-structure, we say that $\Th(M)$ is \textbf{$\d$-computable} if there is a $\d$-algorithm such that, for any $k\in \ulcorner \Sent_L\urcorner$ and any rational $\epsilon>0$, returns an interval $I$ with rational endpoints such that $|I|<\epsilon$ and $\varphi_k^M\in I$.
\end{enumerate}
\end{defn}

This translation is not perfect as far as relative computability goes: 

\begin{lemma}
If $M$ is an $L$-structure, then $\Th(M)$ is $\d$-computable if and only if $\ulcorner \Th(M)\urcorner$ is $\Pi_1^\d$.\footnote{As above, when writing $\ulcorner \Th(M)\urcorner$, we are identifying $\Th(M)$ with its zeroset.}
%
%
\end{lemma}

\begin{proof}
First suppose that $\Th(M)$ is $\d$-computable and consider $j:\mathbb N^2\to \mathbb N$ which, upon input $(k,n)$, if $k=\ulcorner \sigma \urcorner$, uses $\d$ to calculate an interval $I_n$ with $|I_n|<\frac{1}{n}$ and with $\sigma^M\in I_n$, and then returns $1$ if $0\in I_n$ and otherwise returns $0$.  (If $k$ is not the G\"odel number of a sentence, set $j(k,n)=0$.)  Note that $j$ is a $\d$-computable function.  Let $h(k)$ be the least $n$ such that $j(k,n)=0$ when such $n$ exists and be undefined otherwise.  Thus, $h$ is a $\d$-computable partial function.  Since $\operatorname{dom}(h)\cap \ulcorner \operatorname{Sent}\urcorner=\ulcorner \operatorname{Sent}\urcorner\setminus \ulcorner \Th(M)\urcorner$, it follows that $\ulcorner \Th(M)\urcorner$ is $\Pi_1^\d$.  (This uses a different reformulation of $\Sigma_1^\d$ sets, namely the domains of $\d$-computable partial functions.)

Conversely, suppose that $\ulcorner \Th(M)\urcorner$ is $\Pi_1^\d$.   Fix $\sigma \in \Sent_L$ and rational $\epsilon>0$.  Using $\d$, one can enumerate $\ulcorner \Sent_L \urcorner \setminus \ulcorner \Th(M)\urcorner$.  Whenever one sees $\ulcorner \sigma\dotminus r\urcorner$ in this enumeration, one knows that $\sigma^M>r$.  Likewise, whenever one sees $\ulcorner s\dotminus \sigma\urcorner$, one knows that $\sigma^M<s$.  Thus, after some amount of time, one will arrive at such a pair $(r,s)$ with $s-r<\epsilon$ and the interval $(r,s)$ contains $\varphi^M$.
\end{proof}

\begin{rmk}
The above proof can easily be modified to show that if $\ulcorner \Th(M)\urcorner$ is $\Sigma_1^\d$, then $\Th(M)$ is $\d$-computable, and, consequently, that $\ulcorner \Th(M)\urcorner$ is actually $\Delta_1^\d$.  It is unclear if one can show that $\Th(M)$ being $\d$-computable implies $\ulcorner \Th(M)\urcorner$ is $\Delta_1^\d$.
\end{rmk}

There is a proof system for continuous logic presented in \cite{completeness}.  There one defines the relation $ T\vdash \sigma$ for $T$ an $L$-theory and $\sigma$ a restricted $L$-sentence.  It follows that if $\ulcorner T\urcorner $ is $\Sigma_1^\d$, then $\{\ulcorner\sigma\urcorner \ : \ T\vdash \sigma\}$ is also $\Sigma_1^\d$.  

We will also need the following form of the Completeness Theorem for continuous logic; here $\mathbb D$ denotes the dyadic rational numbers.

\begin{fact}[Pavelka-Style Completeness \cite{completeness}]
For any $\sigma\in \Sent_L$, one has
$$\sup\{\sigma^M \ : \ M\models\sigma\}=\inf\{r\in \mathbb D^{>0} \ : \ T\vdash \sigma\dotminus r\}.$$
\end{fact}

%

\noindent The theory $T$ is \textbf{consistent} if there is a sentence $\sigma$ such that $T\not\vdash\sigma$.

\begin{lemma}
$T$ is consistent if and only if $T\not\vdash (1\dotminus \sup_x d(x,x))\dotminus \frac{1}{2}$.
\end{lemma}

\begin{proof}
The backwards direction is immediate.  For the forwards direction, if $T$ is consistent, then by the Completeness Theorem, $T$ is satisfiable (that is, has a model), whence $\sup\{(1\dotminus \sup_x d(x,x))^M \ : \ M\models T\}=1$ and thus $T\not\vdash (1\dotminus \sup_x d(x,x))\dotminus \frac{1}{2}$.
\end{proof}

We set $\Phi_{\Theorem}(X,p):=\Phi_{\Sent}(X)\wedge \varphi_{\Sent_L}(p)\wedge T_X\vdash \varphi_p$, an arithmetic formula.  We also set $\Con(X)$ to be the statement $$\neg\Phi_{\Theorem}(X,\ulcorner (1\dotminus \sup_x d(x,x))\dotminus \frac{1}{2}\urcorner).$$ Thus, $\Con(X)$ is an arithmetical formula and, by the previous lemma, $\Con(\ulcorner T\urcorner)$ is true if and only if $T$ is consistent.

\section{Algebras with computable presentations}

\subsection{Presentations of operator algebras}

In this subsection, we present the basic definitions from computable structure theory for metric structures \cite{FMcN} (building upon the work in \cite{CS1},\cite{CS2}, \cite{CS3}, \cite{CS4}, \cite{CS5}) in the context of operator algebras.

Throughout this subsection, $M$ denotes either a separable \cstar-algebra or a separable tracial von Neumann algebra whose unit ball is denoted by $M_1$.  Given $x,y\in M_1$, a \textbf{rounded combination} of $x$ and $y$ is an element of the form $\lambda x+\mu y$, where $\lambda,\mu\in \mathbb C$ satisfy $|\lambda|+|\mu|\leq 1$.  The rounded combination will be called \textbf{rational} if $\lambda$ and $\mu$ belong to $\mathbb Q(i)$.   


\begin{defn}

\

\begin{enumerate}
\item Given $A\subseteq M_1$, we let $\langle A\rangle$ be the smallest subset of $M_1$ containing $A$ and closed under rational rounded combinations, multiplication, and adjoint.\footnote{In order to fit our discussion under the more general presentation found in \cite{FMcN}, we need our operations to be uniformly continuous, whence the need to restriction attention to operator norm bounded balls.}
\item We say that $A$ \textbf{generates} $M$ if $\langle A\rangle$ is dense in $M_1$ (where density is with respect to the operator norm in case $M$ is a \cstar-algebra and with respect to the 2-norm in case $M$ is a tracial von Neumann algebra).
\item A \textbf{presentation} of $M$ is a pair $M^\#:=(M,(a_n)_{n\in \mathbb N})$, where $\{a_n \ : \ n\in \mathbb N\}\subseteq M_1$ generates $M$.  Elements of the sequence $(a_n)_{n\in \mathbb N}$ are referred to as \textbf{special points} of the presentation while elements of $\langle \{a_n \ : \ n\in \mathbb N\}\rangle$ are referred to as \textbf{rational points} of the presentation.
\end{enumerate}
\end{defn}



The following remark is crucial for what follows:

\begin{rmk}
Given a presentation $M^\#$ of $M$, it is possible to computably enumerate the rational points of $M^\#$.\footnote{Technically, one is computably enumerating ``codes'' for rational points.}  Consequently, it makes sense to consider algorithms which take rational points of $M^\#$ as inputs and/or outputs.
\end{rmk}


\begin{defn}
If $M^\#$ is a presentation of $M$ and $\d$ is a Turing degree, then $M^\#$ is a \textbf{$\d$-computable presentation} if there is a $\d$-algorithm such that, upon input rational point $p\in M^\#$ and $k\in \mathbb N$, returns a rational number $q$ such that $|\|p\|-q|<2^{-k}$ in case $M$ is a \cstar-algebra and $|\|p\|_2-q|<2^{-k}$ in case $M$ is a tracial von Neumann algebra.
\end{defn}

%

The following theorem appears in \cite{CGM}:

\begin{thm}\label{presentation}
Suppose that $M$ is a separable metric structure with a $\d$-computable presentation.  We let $M_C$ be the expansion of $M$ to an $L(C)$-structure such that the new constants are interpreted by the elements of the presentation.  Suppose that $n\geq 1$ and $\bowtie \in \{<,\leq,>,\geq\}$.  Set $$X_n^{\bowtie}:=\{(k,r)\in \mathbb N\times \mathbb {Q}^{>0} \ : \ k\in \ulcorner \forall_{2n}\text{-}\Sent_L\urcorner \text{ and }\varphi_k^M\bowtie r\}.$$  Then:
\begin{itemize}
\item $X_n^{\leq}$ is $\Pi_n^\d$.
\item $X_n^{<}$ is $\Sigma_{n+1}^\d$.
\item $X_n^{\geq}$ is $\Pi_{n+1}^\d$.
\item $X_n^{>}$ is $\Sigma_{n}^\d$.
\end{itemize}
\end{thm}

%

\begin{cor}\label{upperbound}
If $M$ has a $\d$-computable presentation, then $\Th(M)\leq_T \d^{(\omega)}$.
\end{cor}

\subsection{Theories of group operator algebras}

We now apply the previous corollary to group operator algebras.  Fix a finitely generated group $\Gamma$, say generated by the finite set $X$.  We let $\mathbb F_X$ denote the free group on the alphabet $X$ and we fix an effective enumeration $(g_n)$ of the group algebra $\mathbb Q(i)\mathbb F_X$ and let $\pi:\mathbb Q(i)\mathbb F_X\to \mathbb Q(i)\Gamma$ denote the canonical surjection.

We first consider the case of the universal group \cstar-algebra $C^*(\Gamma)$.  Following \cite{canyou}, we let $\|\cdot\|_u$ denote the operator norm on $\mathcal B(\mathcal H_u)$, where $\Gamma\hookrightarrow \mathcal U(\mathcal H_u)$ is the universal representation of $\Gamma$.  Let $B:=\{g_n \ : \ \|\pi(g_n)\|_u\leq 1\}$ and let this be enumerated as $g_{n_k}$.  Since the operators $\pi(g_{n_k})$ are dense in the unit ball of $C^*(\Gamma)$), we have a presentation of $C^*(\Gamma)$ which we refer to as the \emph{standard presentation}.  (This technically depends on the choice of finite generating set and the effective enumeration, but that will not affect what is to follow.)

We will need the following, which is \cite[Corollary 2.2]{canyou}\footnote{This is technically done for $\mathbb Z\Gamma$, but the same proof works for $\mathbb Q(i)\Gamma$.}:

\begin{fact}
Suppose that $\Gamma$ is finitely presented.  Then there is an algorithm which, upon input $n$, computes a decreasing sequence of rational numbers that converges to $\|\pi(g_n)\|_u$.
\end{fact}

\begin{thm}
Suppose that $\Gamma$ is a finitely presented group.  Then the standard presentation of $C^*(\Gamma)$ is $\mathbf 0'$-computable.  Consequently, $\Th(C^*(\Gamma))\leq_T \mathbf 0^{(\omega)}$.
\end{thm}

\begin{proof}
Since the sequence $(g_{n_k})$ is closed under all the symbols, it suffices to show that one can compute $\|\pi(g_{n_k})\|_u$ uniformly in $k$ from $\mathbf 0'$.  Given $k$, by the previous fact, one can use $\mathbf 0'$ to determine $n_k$ and then one can use the previous fact and $\mathbf 0'$ again to compute $\|\pi(g_{n_k})\|_u$.
\end{proof}

We now turn to the reduced group \cstar-algebra $C^*_\lambda(\Gamma)$.  This time, we let $\|\cdot\|_\lambda$ denote the operator norm on $\mathcal B(\ell^2(\Gamma))$, where $\Gamma\hookrightarrow \mathcal U(\ell^2(\Gamma))$ is the left-regular representation.  We now set $B:=\{g_n \ : \ \|\pi(g_n)\|_\lambda\leq 1\}$ and once again let this be enumerated as $g_{n_k}$.  Since the operators $\pi(g_{n_k})$ are dense in the unit ball of $C^*_\lambda(\Gamma)$ (), we have a presentation of $C^*_\lambda(\Gamma)$ which we refer to as the \emph{standard presentation}.

We let $\tau$ denote the canonical trace on $\mathbb C\Gamma$.  We will need the following fact \cite[Lemma 3.1]{canyou}:

\begin{fact}
For any $a\in \mathbb C\Gamma$, we have $\|a\|_\lambda=\sup\{\tau((a^*a)^n)^{1/2n} \ : \ n\in \mathbb N\}$.
\end{fact}

\begin{thm}
Suppose that $\Gamma$ is a finitely generated group with word problem solvable from the Turing degree $\d$.  Then the standard presentation of $C^*_\lambda(\Gamma)$ is $\d'$-computable.  Consequently, $\Th(C^*_\lambda(\Gamma))\leq_T \d^{(\omega)}$.  In particular, if $\Gamma$ has solvable (or even arithmetic) word problem, then $\Th(C^*_\lambda(\Gamma))\leq_T \mathbf 0^{(\omega)}$.
\end{thm}

\begin{proof}
By the assumption on the solvability of the word problem, the function $n\mapsto \tau(\pi(g_n)):\mathbb N\to \mathbb Q$ can be calculated using $\mathbf d$.  Consequently, using the previous fact, we once again see that $n_k$ can be computed uniformly in $k$ from $\mathbf d'$ and that $\|\pi(g_{n_k})\|_\lambda$ can be computed from $\d'$.
\end{proof}


We finally mention the corresponding fact for $L(\Gamma)$.  The only wrinkle here is that we need to use $\|\cdot\|_\lambda$ to determine that elements belong to the operator norm unit ball but then use $\tau$ to compute the 2-norm $\|\cdot\|_{2,\tau}$.  Nevertheless, the above arguments show:  

\begin{thm}
Suppose that $\Gamma$ is a finitely generated group with word problem solvable from the Turing degree $\d$.  Then the standard presentation\footnote{Defined exactly as for $C^*_\lambda(\Gamma).$} of $L(\Gamma)$ is $\d'$-computable.  Consequently, $\Th(L(\Gamma))\leq_T \d^{(\omega)}$.  In particular, if $\Gamma$ has solvable (even arithmetic) word problem, then $\Th(L(\Gamma))\leq_T \mathbf 0^{(\omega)}$.
\end{thm}


\begin{rmk}
When $\Gamma$ is an amenable, ICC group (whence $L(\Gamma)\cong \R$) with solvable word problem, the previous theorem implies that $\R$ has a $\mathbf 0'$-computable presentation.  However, in \cite{GoldHart}, we claimed that $\R$ has a computable presentation.  Since no proof was offered there, we say a few words here.

Indeed, one can effectively enumerate $\bigcup_n M_n(\mathbb Q(i))$ as $(A_n)$.  Since, given a matrix $A$, one has $\|A\|^2=\|A^*A\|$ and the latter equals the largest eigenvalue of $A^*A$, using familiar algorithms from numerical analysis, one can compute, given $n$, a sequence $p_{m,n}$ of rational upper bounds for $\|A_n\|$.  By letting $B_{m,n}:=\frac{1}{p_{m,n}}A_n$, one obtains an effectively enumerated dense sequence from the unit ball of $\R$.  It is then routine to effectively calculate the 2-norm $\|\cdot\|_2$ of $B_{m,n}$ using the normalized trace on matrix algebras. 
\end{rmk}

\begin{rmk}
As a consequence of the fact that $\R$ has a computable presentation, Corollary \ref{upperbound} implies that the universal theory of $\R$ can be computed from $\mathbf 0'$.  In \cite{notcomputable}, the current authors show that the universal theory of $\R$ is not computable and in fact $\mathbf 0'$ can be computed from the universal theory of $\R$.
\end{rmk}

It is not clear to us whether some fundamental \cstar-algebras such as the Cuntz algebra $\mathcal O_2$ have computable (or even hyperarithmetic) presentations.  Thus, in order to deduce that such an algebra has a hyperarithmetic theory, we need a new technique, which is the content of the next section.

%


\section{Finite forcing companions}

\subsection{Model-theoretic forcing:  a recap}


In this section, we summarize the relevant background on building models by games in the continuous setting from \cite{Gold}.

We fix an $L$-theory $T$.  A \textbf{pre-condition} is a finite set $p$ of expressions of the form $\varphi<r$, where $\varphi\in \Sent_{L(C)}$ is quantifier-free and $r$ is a positive rational number.  A pre-condition $p$ is a \textbf{condition} (for $T$) if $T\cup p$ is satisfiable.

We consider a two-person game involving the players $\forall$ and $\exists$.  Players $\forall$ and $\exists$ take turns playing conditions extending the previous players move.  Thus, $\forall$ starts by playing the condition $p_0$, whence $\exists$ follows up by playing the condition $p_1\supseteq p_0$, and then $\forall$ follows that play with some condition $p_2\supseteq p_1$, etc... After $\omega$ many steps, the two players have together played a chain $p_0\subseteq p_1\subseteq p_2\subseteq\cdots$ of conditions whose union we will denote by $\bar p$.  

We call the above play \textbf{definitive} if, for every atomic $L(C)$-sentence $\varphi$, there is a unique $r\in [0,1]$ such that $T\cup \bar p \models \varphi=r$.  In this case, $\bar p$ describes an $L(C)$-prestructure $A_0^+(\bar p)$ whose completion will be denoted by $A^+(\bar p)$ and will be referred to as the \textbf{compiled structure}.  The reduct of $A^+(\bar p)$ to $L$ will be denoted by $A(\bar p)$.  If $\bar p$ is clear from context, we will denote $A^+(\bar p)$ and $A(\bar p)$ simply by $A^+$ and $A$ respectively.

Note that, regardless of player $\forall$'s moves, player $\exists$ can always ensure that the play of the game is definitive.

\begin{defn}
Let $P$ be a property of $L(C)$-structures.  The game $G(P)$ is the game whose moves are as above and such that Player $\exists$ wins $G(P)$ if and only if $\bar p$ is definitive and $A^+(\bar p)$ has property $P$.  We say that $P$ is \textbf{enforceable} if Player $\exists$ has a winning strategy in $G(P)$.  We say that the condition $p$ \textbf{forces} $P$, denoted $p\Vdash P$, if, for any position $(p_0,\ldots,p_k)$  of the game $G(P)$, if $p\subseteq p_k$, then the position is winning for $\exists$ in $G(P)$.  We write $\Vdash P$ when $\emptyset\Vdash p$.
\end{defn}

\begin{fact}[Conjunction Lemma]
If $p$ forces $P_i$ for each $i<\omega$, then $p$ forces the conjunction $\bigwedge_{i<\omega} P_i$..
\end{fact}

\begin{fact}\label{enforceuniversal}
It is enforceable that the compiled structure be a model of $T_\forall$.
\end{fact}

For a condition $p$ and an $L(C)$-sentence $\varphi$, we set $$F_p(\varphi):=\inf\{r \ : \ p\Vdash \varphi<r\}.$$  

\begin{lemma}
For any condition $p$ and $L(C)$-sentence $\varphi$, $F_p\left(\varphi\right)\leq r$ if and only if $p\Vdash \varphi\leq r$.
\end{lemma}

\begin{proof}
First suppose that $F_p(\varphi)\leq r$.  Then for every $n>1$, $p\Vdash \varphi<r+\frac{1}{n}$, whence, by the Conjunction Lemma, $p\Vdash \varphi\leq r$.

Conversely, suppose $p\Vdash \varphi\leq r$.  Then for every $s>r$, we have $p\Vdash \varphi<s$, so $F_p(\varphi)\leq s$.  
\end{proof}

We will require the following fact about the function $F_p$, whose proof can be found in \cite[Fact 2.21 and Theorem 2.22]{Gold}:

\begin{fact}
For any condition $p$ and $L(C)$-formula $\varphi(x)$, we have
$$F_p\left(\inf_x\varphi(x)\right)=\sup_{q\supseteq p}\inf_{q'\supseteq q}\inf_{c\in C}F_{q'}(\varphi(c)).$$
\end{fact}

\subsection{Finitely generic theories are hyperarithmetic}

Until further notice, $T$ is an $L$-theory that is $\Sigma_1^\d$ for some Turing degree $\d$.  We let $\Pre_T$ denote the set of numbers $e\in \mathbb N$ such that $e$ codes a tuple $(e_1,\ldots,e_n)$, where each $e_i$ codes a pair $(k,r)$ with $k$ the code of a quantifier-free restricted $L(C)$-sentence and $r\in \mathbb D^{>0}$.  We let $\Cond_T$ denote those $e\in \Pre_T$ such that $e$ codes a condition for $T$.  If $e\in \Pre_T$, we let $p_e$ denote the pre-condition it codes.

\begin{lemma}

\

\begin{enumerate}
\item $\Pre_T$ is computable.
\item The set $\{(e_1,e_2)\in \Pre_T \ : \ p_{e_1}\subseteq p_{e_2}\}$ is computable.
\item $\Cond_T$ is $\Pi_1^\d$.
\end{enumerate}
\end{lemma}

\begin{proof}
The first two statements are clear.  Since $\{\varphi_{k_i}<r_i \ : \ i=1,\ldots, n\}$ is a condition for $T$ if and only if $\neg\Phi_{\operatorname{Theorem}}(\ulcorner T\urcorner,\ulcorner \min_{i=1,\ldots,n}r_i\dotminus \varphi_{k_i}\urcorner)$, it follows that $\Cond_T$ is $\Pi_1^\d$. 
\end{proof}

Given a pre-condition $p=\{\varphi_i<r_i \ : \ i=1,\ldots,n\}$, let $$p^\$:=\bigcup\left\{\max_{1\leq i\leq n}(\varphi_i\dotminus s_i) \ : \ s_1,\ldots,s_n\in \mathbb D, \ s_i<r_i\right\}.$$  Note that $p^\$$ is an $L(C)$-theory for which $A\models T\cup p$ if and only if there is $\theta\in p^\$$ such that $A\models T\cup \theta$.

The following lemma is clear:

\begin{lemma}
There is a computable set $C\subseteq \mathbb N^2$ such that $(e,c)\in C$ if and only if $e\in \Pre_T$, $c\in \Sent_{L(C)}$, and $\varphi_c\in \ulcorner p^\$\urcorner$.
\end{lemma}


\begin{prop}
Suppose that $p$ is a condition for $T$ and $\psi(x)\in \Sent_{L(C)}$ is quantifier-free.  Then
$$p\Vdash \sup_x \psi(x)\leq r \Leftrightarrow  (\forall \theta\in p^\$)T\cup \{\theta\} \vdash \sup_x\psi(x)\dotminus r.$$
\end{prop}

\begin{proof}
First suppose that $\theta\in p^\$$ is such that $T\cup \{\theta\}\not\vdash (\sup_x \psi(x))\dotminus r$.  Then by the Completeness Theorem, there is an $L(C)$-structure $A$ such that $A\models T\cup \{\theta\}$ and $(\sup_x \psi(x))^A>r$.  Take constants $c$ from $C$ such that $\psi(c)^A>r$.  Then $q:=p\cup\{1\dotminus \psi(c)<1-r\}$ is a condition and $q\Vdash \psi(c)>r$.  Consequently, $p\not\Vdash \sup_x \psi(x)\leq r$.

Conversely, suppose that $T\cup \{\theta\}\vdash (\sup_x \psi(x))\dotminus r$ for all $\theta\in p^\$$.  Suppose that player I plays $p$.  Then player II can play so that the compiled structure $A$ is such that $A\models T_\forall \cup p$.  Take $B\models T$ such that $A\subseteq B$ and take $\theta\in p^\$$ such that $B\models \theta$.  By assumption and the Completeness Theorem again, $(\sup_x \psi(x))^B\leq r$, whence $(\sup_x \psi(x))^A\leq r$.  It follows that $p\Vdash (\sup_x\psi(x))\leq r$, as desired.

\end{proof}

\begin{thm}
For each $n\geq 0$, the set $$\{(p,k,r) \ : \ e\in \Cond_T, \ k\in \ulcorner \forall_{2n+1}\text{-}\Sent_{L(C)}\urcorner, \text{ and }F_{p_e}(\varphi_k)\leq r\}$$ is $\Pi_{2n+2}^\d$.
\end{thm}

\begin{proof}

First suppose that $n=0$.  By the previous lemma, $F_{p_e}(\varphi_k)\leq r$ if and only if:  for every $c\in \mathbb N$, either $(e,c)\notin C$ or else $\Phi_{\operatorname{Theorem}}(\ulcorner T\urcorner\cup\{c\},\ulcorner \varphi_k\dotminus r\urcorner)$.  This condition is $\Pi_2^\d$.

%
%

Now suppose $n\geq 1$ and $\varphi=\sup_x\inf_y\psi(x,y)$, where $\psi\in \forall_{2n-1}\text{-}\Sent_{L(C)}$.  Further suppose that $p$ is a condition for $T$.  Then the following are equivalent:
\begin{itemize} 
\item $F_p(\varphi)\leq r$.
\item $p\Vdash \varphi\leq r$.
\item $(\forall c\in C)$ $p\Vdash \inf_y \psi(c,y)\leq r$.
\item $(\forall c\in C)$ $F_p(\inf_y \psi(c,y)\leq r$.
\item $(\forall c\in C)(\forall q\supseteq p)(\forall m\in \mathbb N)(\exists q'\supseteq q)(\exists d\in C)F_{q'}(\psi(c,d))\leq r+\frac{1}{m}$. 
\end{itemize}
By induction, the ``matrix'' in the last bullet is $\Pi_{2n}$ in the parameters; since the G\"odel code for $\psi$ is computable from the G\"odel code for $\varphi$, we see that the entire condition is $\Pi_{2n+2}^\d$, as desired.
\end{proof}

Recall that $T$ has the \emph{joint embedding property (JEP)} if any two models of $T$ can be jointly embedded into a third model of $T$.  Before moving on, we need to recall:

\begin{fact}[\cite{Gold}]
If $T$ has the JEP, then for any $L$-sentence $\sigma$, there is a unique real number $r_\sigma$ such that $\Vdash \sigma=r$.  
\end{fact}

\begin{defn}
If $T$ has the JEP, then the \textbf{finite forcing companion of $T$} is the complete theory $T^f:=\{|\sigma-r_\sigma| \ : \ \sigma\in \Sent_L\}$.
\end{defn}

\begin{cor}
Suppose that $T$ has the JEP.  Then $T^f\leq_T \d^{(\omega)}$.
\end{cor}

\begin{proof}
Fix $k\in \ulcorner \forall_{2n+1}\text{-}\Sent_L\urcorner$ and $\epsilon>0$.  Given positive $r\in \mathbb D$, using $\d^{(2n+2)}$, one can determine whether or not $F_\emptyset(\varphi_k)\leq r$, that is, whether or not $\Vdash \varphi_k\leq r$.  If the answer is ``yes'', then one concludes that $\varphi_k^{T^f}\leq r$.  If the answer is ``no'', then since we know that $\Vdash \varphi_k=\varphi_k^{T^f}$, then we have that $\varphi_k^{T^f}>r$.  Thus, one runs this algorithm over all possible rational $r>0$ and waits until one sees that $\varphi_k^{T^f}$ belongs to an interval of radius $\epsilon$.
\end{proof}

\subsection{Examples}

We now give some applications of the previous corollary:

\begin{example}
Let $T$ be the theory of \cstar-algebras.  Then $T$ is $\Sigma_1$ and has the JEP.  It follows that $T^f\leq_T \mathbf 0^{(\omega)}$.  If the so-called \emph{Kirchberg embedding problem} (KEP) has a positive solution, that is, every \cstar-algebra embeds into an ultrapower of the Cuntz algebra $\mathcal O_2$, then it was shown in \cite{Gold} that $T^f=\Th(\mathcal O_2)$ and thus $\Th(\mathcal O_2)\leq_T \mathbf 0^{(\omega)}$.

Alternatively, if one lets $T=\Th_\forall(\mathcal O_2)$, the universal theory of $\mathcal O_2$, then $T$ has the JEP and $T^f=\Th(\mathcal O_2)$.  Consequently, if $\Th_\forall(\mathcal O_2)$ is $\Sigma_n^\d$ for some $n$, then $\Th(\mathcal O_2)\leq \d^{(\omega)}$.  As in \cite{GoldHart}, one can show that KEP implies that $\Th_\forall(\mathcal O_2)$ is computable, whence we arrive at the same conclusion as in the previous paragraph.
\end{example}


\begin{example}
If $T$ is the theory of commutative \cstar-algebras, then $T$ is $\Sigma_1$ and has the JEP.  In this case, $T^f$ equals the model companion of $T$, namely $\Th(C(2^\omega))$, and thus we conclude $\Th(C(2^\omega))\leq_T \mathbf 0^{(\omega)}$. 
\end{example}

Since $\Th(C(2^\omega))$ is arguably the nicest theory of an infinite-dimensional operator algebra, it is not outlandish to ponder the following:

\begin{question}
Is $\Th(C(2^\omega))$ computable?
\end{question}

\begin{example} If $T$ is the theory of projectionless commutative \cstar-algebras, then $T$ is $\Sigma_1$\footnote{See \cite[Remark 1.2]{pseudoarc}.} and has the JEP.  In \cite{Gold}, it was shown that $T^f=\Th(C(\mathbb P))$, where $\mathbb P$ is the so-called \emph{pseudoarc}.  It follows that $\Th(C(\mathbb P))\leq_T \mathbf 0^{(\omega)}$.
\end{example}

Although this paper is about operator \emph{algebras}, we would be remiss if we did not mention the following examples:

\begin{example}
Suppose that $T$ is the theory of operator spaces (resp. operator systems).  Then $T$ is $\Sigma_1$\footnote{See \cite[Appendix B]{KEP}.} and has the JEP.  In \cite{Gold}, it was shown that $T^f=\Th(\mathbb {NG})$ (resp. $T^f=\Th(\mathbb {GS})$), where $\mathbb {NG}$ is the so-called \emph{noncommutative Gurarij space} (resp. $\mathbb {GS}$ is the \emph{Gurarij operator system}).  Consequently, $\Th(\mathbb {NG}),\Th(\mathbb {GS})\leq_T \mathbf 0^{(\omega)}$.
\end{example}


\subsection{Musings on oracle computability of presentations of compiled structures}

In the previous subsection, we saw that having hyperarithmetic theory is an enforceable property.  However, what about adding parameters?

First, we recall that a model $A$ of $T$ is \textbf{enforceable} if the property that the reduct of the compiled structure be isomorphic to $A$ is an enforceable property.

\begin{prop}
Suppose that $T$ has JEP.  Then the following are equivalent:
\begin{enumerate}
\item There is an oracle $\d$ such that it is enforceable that the elementary diagram (with respect to some presentation) of the compiled structure is $\d$-computable.
\item There is an oracle $\d$ such that it is enforceable that the compiled structure has a $\d$-computable presentation.
\item There is an enforceable model of $T$.
\end{enumerate}
\end{prop}

\begin{proof}
The implications (1) implies (2) and (3) implies (1) are obvious.  The implication (2) implies (3) follows from the so-called \emph{Dichotomy Theorem} (see \cite{Gold}) and the fact that only countably many models of $T$ can have a $\d$-computable presentation for any given oracle $\d$.
\end{proof}

\begin{rmk}
At first glance, it seems that, given an oracle $\d$, one can use a diagonalization argument to prevent the compiled structure from having a $\d$-computable presentation.  However, that strategy only prevents the ``canonical'' presentation of the compiled structure (that is, the one given in terms of the constants from $C$) from being $\d$-computable.
\end{rmk}

\begin{rmk}
In the case of II$_1$ factors, \cstar-algebras, and stably finite \cstar-algebras, it is unknown if there is an enforceable model.  However, in the case of II$_1$ factors, we conjecture that having an enforceable model is equivalent to the truth of the Connes Embedding Problem (CEP).  (In \cite{Gold}, it was shown that CEP is equivalent to the hyperfinite II$_1$ factor itself being enforceable.)  Recently, a purported proof of the failure of CEP has appeared \cite{quantum}.  If this proof is correct and the aforementioned conjecture holds, then this would show that, given any oracle $\d$, it is not enforceable that the compiled II$_1$ factor have a $\d$-computable presentation.
\end{rmk}



\section{Degrees of theories of existentially closed models}

Fix an $L$-theory $T$.  We recall the following definition:

\begin{defn}
The $L$-structure $A$ is an \textbf{existentially closed (e.c.) model of $T$} if $A$ is a model of $T$ and for any $B\models T$ with $A\subseteq B$ and any existential $L_A$-sentence, one has $\varphi^A=\varphi^B$.
\end{defn}

There is a lot to say about e.c. operator algebras (see \cite{ecfactor} and \cite{KEP} for example).  In this section, we are concerned with the complexity of $\ulcorner \Th(M)\urcorner$ for $M$ an e.c. model of $T$.  Our main result is the following, which is a continuous analog of a theorem of Simpson:

\begin{thm}\label{manyectheories}
Suppose that $\ulcorner T\urcorner$ is arithmetic.  If there is an e.c. model $M$ of $T_\forall$ such that $\ulcorner\Th(M)\urcorner$ is not hyperarithmetic, then there are continuum many different theories of e.c. models of $T_\forall$.
\end{thm}

Our approach follows \cite[Chapter 7, Section 4]{HW}.  For the rest of this section, we assume that $\ulcorner T\urcorner$ is arithmetic.

The following definitions are taken from \cite{completeness}:



\begin{defn}
Suppose $T'$ is an $L$-theory.  We say that $T'$ is:
\begin{enumerate}
\item \textbf{maximal consistent} if it is consistent and:
\begin{enumerate}
\item if $\sigma\in \Sent_L$ is such that $\sigma\dotminus 2^{-n}\in T'$ for all $n$, then $\sigma\in T'$, and
\item for any two sentences $\sigma,\tau\in \Sent_L$, either $\sigma\dotminus \tau\in T'$ or $\tau\dotminus \sigma\in T'$.
\end{enumerate}
\item \textbf{Henkin} if for every $\varphi(x)\in \operatorname{Form}_L$, $p,q\in \mathbb D$ with $p<q$, there is a constant symbol $c$ such that $\min(\sup_x\varphi(x)\dotminus q,p\dotminus \varphi(c))\in T'$.
\end{enumerate}
\end{defn}

In the proof of the next lemma, we ask the reader to recall the definitions of the computable functions $f$ and $g$ from Lemma \ref{fandg}.

\begin{lemma}
There is an arithmetic formula $\Phi_0(X)$ such that, for $S\subseteq \mathbb N$, one has $\Phi_0(S)$ iff $S=\ulcorner T'\urcorner$ for some maximal consistent Henkin $L(C)$-theory $T'$ such that $T_\forall\subseteq T'_\forall$.
\end{lemma}

\begin{proof}
Consider the following formulae:
\begin{itemize}
\item $\Psi_1(X)$ is $(\forall p)((\varphi_{\Sent_{L(C)}}(p)\wedge (\forall n)\Phi_{\Theorem}(X,f(p,n)))\rightarrow p\in X)$.
\item $\Psi_2(X)$ is $(\forall p,q)(g(p,q)\in X \vee g(q,p)\in X)$.
\item $\Psi_3(X)$ is $(\forall p)(\forall a,b\in \mathbb D)(a<b\rightarrow (\exists n)(\ulcorner \min(\sup_x\varphi_p(x)\dotminus b,a\dotminus \varphi_p(c_n))\urcorner\in X)$.
\item $\Psi_4(X)$ is $(\forall p)(\varphi_{\forall\text{-}\Sent_L}(p)\wedge \Phi_{\Theorem}(\ulcorner T\urcorner,p))\to p\in X)$.
\end{itemize}
Since $\ulcorner T\urcorner$ is arithmetical, it follows that $\Psi_4(X)$ is arithmetical, and thus $$\Phi_0(X):=\Phi_{\Sent_{L(C)}}(X)\wedge \Con(X)\wedge \bigwedge_{i=1}^4\Psi_i(X)$$ is an arithmetical formula.  It is clear that this $\Phi_0$ is as desired.
\end{proof}

As in classical logic, if $T'$ is a maximal consistent Henkin theory, then there is a canonical model $M_{T'}\models T'$ whose underlying universe is an appropriate quotient of the closed $L$-terms.  We refer the reader to \cite{completeness} for more details.  If $\Phi_0(S)$ holds, say $S=\ulcorner T'\urcorner$ for some maximal consistent Henkin theory $T'$, then we may write $M_S$ instead of $M_{T'}$.


\begin{thm}
There is an arithmetic predicate $\Phi_1(X)$ such that:
\begin{enumerate}
\item If $\Phi_1(S)$ holds, then $\Phi_0(S)$ holds and the $L$-reduct of $M_S$ is an e.c. model of $T_\forall$.
\item If $M$ is a separable e.c. model of $T_\forall$, then there is $S$ such that $\Phi_1(S)$ holds and $M\cong M_S$.
\end{enumerate}
\end{thm}

\begin{proof}
Let $\operatorname{Diag}(p,X)$ be the formula 
$$\Phi_0(X)\wedge \varphi_{\operatorname{qf}_{L(C)}}(p)\wedge p\in X.$$
Let $\operatorname{Consist}(X,p)$ be the formula
$$\Phi_0(X)\wedge \varphi_{\Sent_{L(C)}}(p)\wedge \Con(\ulcorner T_\forall\urcorner\cup\{\phi_p\}\cup\{i \ : \ \operatorname{Diag}(i,X)\}).$$
Since $\ulcorner T\urcorner$ is arithmetic, so is $\ulcorner T_\forall \urcorner$, and thus so is $\operatorname{Consist}(X,p)$.
Let $\Psi_5(X)$ be 
$$\forall p((\operatorname{Consist}(X,p)\wedge \varphi_{\exists\text{-}\Sent_{L(C)}}(p))\rightarrow p\in X).$$
Finally, set $\Phi_1(X):=\Phi_0(X)\wedge \Psi_5(X)$.
If $\Phi_1(S)$ holds, then $M_S$ is clearly an e.c. model of $T_\forall$.  Conversely, suppose that $M$ is a separable e.c. model of $T_\forall$.  Let $A$ be a countable dense subset of $A$ closed under the function symbols and let $M_A$ be the expansion of $M$ to an $L(C)$-structure in the obvious way.  It is clear that $\operatorname{Th}(M_A)$ is a maximal consistent Henkin theory.  Let $S:=\ulcorner \operatorname{Th}(M_A)\urcorner$.  Since $M_S\cong M_A$, we have that $\Phi_1(S)$ holds.
\end{proof}

\begin{cor}\label{sigmadefinable}.
$\{\ulcorner \operatorname{Th}(M) \urcorner \ : \ M\text{ is an e.c. model of }T\}$ is a $\Sigma_1^1$ subset of $\mathcal P(\mathbb N)$.
\end{cor}

\begin{proof}
Let $\Phi_2(X):=\exists Y(\Phi_1(Y)\wedge \forall p(p\in X\leftrightarrow (p\in Y \wedge \varphi_{\Sent}(p)))$.  Then $\Phi_2$ is a $\Sigma_1^1$ formula and defines the set in question.
\end{proof}

The following can be found in \cite[Theorem IV.4.1.]{Barwise}:

\begin{fact}\label{CH}
If $X\subseteq \mathcal P(\mathbb N)$ is a $\Sigma_1^1$-definable set and contains a nonhyperarithmetic set, then $|X|=2^{\aleph_0}$.
\end{fact}

\noindent Theorem \ref{manyectheories} follows immediately from Corollary \ref{sigmadefinable} and Fact \ref{CH}.

At present, it is not known whether or not there are even two distinct theories of e.c. \cstar-algebras or two distinct theories of e.c. tracial von Neumann algebras.  Although we will not go into the definition here, there is a special kind of e.c. model called a \emph{finitely generic} model (see \cite{Gold} for the precise definition).  If $M$ is finitely generic, then $T^f=\Th(M)$, whence is hyperarithmetic if $T$ is hyperarithmetic.  Clearly, if we can find an e.c. \cstar-algebra (or tracial von Neumann algebra) with non-hypearithmetic theory, then its theory must differ from that of the finitely generic algebra.  The previous corollary says that in fact one would obtain continuum many different theories of e.c. algebras.

In the case of groups, full second-order arithmetic Turing reduces to the theory of so-called \emph{infinitely generic} groups (see, e.g. \cite[Section 5.3, Exercise 14]{Hodges}), giving one way to prove that there are continuum many theories of e.c. groups.

\begin{conj}
Second-order arithmetic Turing reduces to both the theory of infinitely generic \cstar-algebras and the theory of infinitely generic tracial von Neumann algebras. 
\end{conj}

If one restricts to so-called \emph{embeddable} tracial von Neumann algebras, that is, those that embed into an ultrapower of $\mathcal R$, then $T^f=\Th(\mathcal R)$.  If the previous conjecture held for this class as well, then this would show that $\Th(\mathcal R)$ is not the theory of infinitely generic embeddable tracial von Neumann algebras, something that was erroneously claimed in \cite{ecfactor}.

Tangentially related to this discussion, we recall that first-order arithmetic Turing reduces to the theory of any e.c. group \cite[Section 3.3, Exercise 8]{Hodges}.  This leads us to:

\begin{conj}
First-order arithmetic Turing reduces to the theory of any (embeddable) e.c. II$_1$ factor.
\end{conj}

If the previous conjecture is true, then since $\R$ is an e.c. embeddable factor (see \cite{ecfactor}), one would have that $\Th(\R)$ is Turing-equivalent to first-order arithmetic, whence the upper bound from Section 3 would be sharp.

In \cite[Question 2.5]{ecfactor}, it was asked whether or not $\Th(\R)$ is $\forall\exists$-axiomatizable.  By Theorem \ref{presentation} above, if $\Th(\R)$ were $\forall_n$-axiomatizable, then $\Th(\R)$ would be arithmetic.  Consequently, if the previous conjecture were true, one would have that $\Th(\R)$ does not admit any quantifier simplification.


%
%

\end{document}